\documentclass[12pt]{article}
\usepackage{fullpage}
\usepackage{amsmath}
\usepackage{amsthm}
\usepackage{amssymb}

\newtheorem{theorem}{Theorem}
\newtheorem{cor}{Corollary}
\newtheorem{proposition}{Proposition}
\newtheorem{lemma}{Lemma}

\begin{document}

\vspace*{30px}

\begin{center}\Large \textbf{Hyperfibonacci Sequences and Polytopic Numbers} 
\bigskip

\large
Ligia L. Cristea \\
Karl-Franzens-Universit\"at Graz\\
Institute for Mathematics and Scientific Computing\\
Heinrichstrasse 36, A-8010 Graz, Austria\\
strublistea@gmail.com\\
and\\
Ivica Martinjak\\
University of Zagreb, Faculty of Science\\
Bijeni\v cka 32, HR-10000 Zagreb, Croatia\\
imartnjak@phy.hr\\
and\\
Igor Urbiha\\
Zagreb University of Applied Sciences (former Polytechnic of Zagreb)\\
Vrbik 8, HR-10000 Zagreb, Croatia\\
igor.urbiha@tvz.hr

\end{center}

\begin{abstract} 
We prove that the difference between the $n$-th hyperfibonacci number of $r$-th generation and its two consecutive predecessors is the $n$-th regular $(r-1)$-topic number. Using this fact we provide an equivalent recursive definition of hyperfibonacci sequences and derive an extension of the Binet formula. We also prove further identities involving both hyperfibonacci and hyperlucas sequences, in full generality.  
\end{abstract}

\section{Introduction}

The hyperfibonacci sequence of $r$-th generation, denoted by $(F_n^{(r)} )_{n \ge 0}$, is defined by the recurrence relation
\begin{eqnarray}\label{hyperFibo}
F_n^{(r)}  = \sum_{k=0} ^n F_k ^{(r-1)}, \enspace F_n^{(0)}= F_n, \enspace F_0^{(r)}=0, \enspace F_1^{(r)}=1,
\end{eqnarray}
where $r \in \mathbb{N}$ and $F_n$ is the $n$-th term of the Fibonacci sequence. In the same manner one can define hyperlucas sequences. The hyperlucas sequence of $r$-th generation $(L_n^{(r)} )_{n \ge 0}$ is defined by means of the recurrence relation
%\begin{eqnarray}%\label{hyperLucas}
\[
L_n^{(r)}  = \sum_{k=0} ^n L_k ^{(r-1)}, \enspace L_n^{(0)}= L_n, \enspace L_0^{(r)}=2, \enspace L_1^{(r)}=2r+1,
\]
%\end{eqnarray}
where $r \in \mathbb{N}$ and $L_n$ is the $n$-th Lucas number. Table \eqref{tab1} shows the starting terms of the first two generations of the hyperfibonacci and hyperlucas sequences.

\begin{table}[ht!] \label{tab1}
\begin{center}
$
\begin{array}{c||ccccccccccc}
n & 0 & 1 & 2 & 3 & 4 & 5 & 6 & 7 & 8 & 9 & 10\\ \hline\hline
F_n^{(1)} & 0 &  1& 2 & 4 & 7 & 12 & 20 & 33 & 54 & 88 & 143 \\
F_n^{(2)} & 0 &  1& 3& 7 &14 & 26 & 46 & 79 & 133 & 221 & 364  \\\hline
L_n^{(1)} & 2& 3 & 6 & 10 & 17 & 28 & 46 & 75 & 122 &198  & 321\\
L_n^{(2)} & 2 &  5& 11 & 21 & 38 &  66&  112&  187& 309 & 507 & 828
\end{array}
$
\caption{The starting terms of the first two generations of the hyperfibonacci and hyperlucas sequences.}
\end{center}
\end{table}

Dil and Mez\H o recently introduced these sequences in a study of a symmetric algorithm for hyperharmonic, Fibonacci and some other integer sequences \cite{DiMe}. The same authors provide further refinements on this subject \cite{DiMe2}. On the other hand, hyperfibonacci sequences occur naturally as the number of {\it board tilings} with {\it squares} and {\it dominoes}. We let $f_{m}^{(r)}$ denote the number of ways to tile an $m$-board with at least $r$ dominoes. Then, for $n, r \ge 0$ the relation
\begin{eqnarray}\label{hyperFiboBoard}
f_{n+2r}^{(r)} = F_{n+1}^{(r)}
\end{eqnarray}
holds, with $f_0^{(r)}:=1$ \cite{BeBe}. Thus, the $(n+1)$-th hyperfibonacci number of $r$-th generation is equal to the number of $(n+2r)$-board tilings with at least $r$ dominoes. Equivalently, hyperfibonacci sequences represent the number of decompositions of an integer into summands $1$ and $2$, with the constraint on the number of $2$s. This follows immediately from the fact that one can code squares and dominoes with $1$ and $2$, respectively. For example, when $n=3$ and $r=1$ we have $f_5^{(1)}=F_{4}^{(1)}=7$, which is the number of decompositions of $5$, 
\begin{eqnarray*}
5&=&2+2+1=2+1+2=1+2+2=2+1+1+1\\
 &=&1+2+1+1+1= 1+1+2+1=1+1+1+2. 
\end{eqnarray*}
Figure \eqref{Fig1} shows the related $5$-board tilings.

It is worth mentioning that several interesting number theoretical and combinatorial properties of these sequences have already been proven \cite{CaZh, LiZa, ZLZ}. In the following we use these facts in order to establish further identities for the hyperfibonacci sequences.

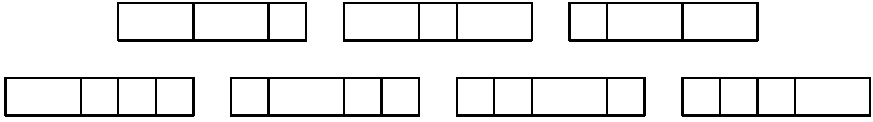
\begin{figure}					
\setlength{\unitlength}{0.5cm}

\begin{picture}(0,3)(-2,0)
\put(3,2){\line(1,0){5}}
\put(5,2){\line(0,1){1}}
\put(7,2){\line(0,1){1}}
\put(8,2){\line(0,1){1}}
\put(3,2){\line(0,1){1}}
\put(3,3){\line(1,0){5}}

\put(9,2){\line(1,0){5}}
\put(11,2){\line(0,1){1}}
\put(12,2){\line(0,1){1}}
\put(14,2){\line(0,1){1}}
\put(9,2){\line(0,1){1}}
\put(9,3){\line(1,0){5}}

\put(15,2){\line(1,0){5}}
\put(16,2){\line(0,1){1}}
\put(18,2){\line(0,1){1}}
\put(20,2){\line(0,1){1}}
\put(15,2){\line(0,1){1}}
\put(15,3){\line(1,0){5}}

\put(0,0){\line(1,0){5}}
\put(2,0){\line(0,1){1}}
\put(3,0){\line(0,1){1}}
\put(4,0){\line(0,1){1}}
\put(5,0){\line(0,1){1}}
\put(0,0){\line(0,1){1}}
\put(0,1){\line(1,0){5}}

\put(6,0){\line(1,0){5}}
\put(7,0){\line(0,1){1}}
\put(9,0){\line(0,1){1}}
\put(10,0){\line(0,1){1}}
\put(11,0){\line(0,1){1}}
\put(6,0){\line(0,1){1}}
\put(6,1){\line(1,0){5}}

\put(12,0){\line(1,0){5}}
\put(13,0){\line(0,1){1}}
\put(14,0){\line(0,1){1}}
\put(16,0){\line(0,1){1}}
\put(17,0){\line(0,1){1}}
\put(12,0){\line(0,1){1}}
\put(12,1){\line(1,0){5}}

\put(18,0){\line(1,0){5}}
\put(19,0){\line(0,1){1}}
\put(20,0){\line(0,1){1}}
\put(21,0){\line(0,1){1}}
\put(23,0){\line(0,1){1}}
\put(18,0){\line(0,1){1}}
\put(18,1){\line(1,0){5}}
\end{picture}

\caption{Combinatorial interpretation of the hyperfibonacci number $F_4^{(1)}$. }\label{Fig1}
\end{figure}

\section{Alternative definition of hyperfibonacci sequences}

One can immediately see that the relation
\begin{eqnarray} \label{eqBasicRechyperFibo}
F_{n}^{(r)} = F_{n-1}^{(r)} + F_{n}^{(r-1)}
\end{eqnarray}
follows from the definition of hyperfibonacci numbers \eqref{hyperFibo}. Now we present a recurrence relation for the $n$-th hyperfibonacci number that involves its two predecessors of the same generation. Lemma \ref{figurateNumbersR1} gives such a relation for $r=1$.

\begin{lemma} \label{figurateNumbersR1} 
	The elements of the sequence $(F_n^{(1)})_{n \ge 0}$ of the first generation of hyperfibonacci numbers satisfy the following recurrence:
\[
F_{n+2}^{(1)}=F_{n+1}^{(1)}+F_{n}^{(1)} + 1.
\]
\end{lemma}

\begin{proof}
According to \eqref{hyperFiboBoard} there are $F_{n+2}^{(1)}$ $(n+3)$-board tilings with at least one domino. We consider the last tile in a tiling, which can be either a square or a domino. Tilings that end with a square are obviously equinumerous to tilings of an $(n+2)$-board having at least one domino. The number of such tilings is $F_{n+1}^{(1)}$. 

However, the number of tilings ending with a domino is not equal to $F_n^{(1)}$ since, when fixing the last domino, here we have one $(n+1)$-tiling with all squares. This tiling does not meet the condition on the minimal number of dominoes in a tiling, so we have to add 1 in order to establish the equality.
\end{proof}

With similar arguments one can prove that in the case $r=2$, the relation 
\[
F_{n+2}^{(2)}=F_{n+1}^{(2)}+F_{n}^{(2)} + n+2
\]
holds. We generalize these recurrences in Lemma \ref{figurateNumbers}. We recall that {\it polytopic numbers} are a generalization of square and triangular numbers. These numbers can be represented by a regular geometric arrangement of equally spaced points. The $n$-th regular $r$-topic number $P_n^{(r)}$ is equal to 
\[
P_n^{(r)}=\binom{n+r-1}{r}.
\]

\begin{lemma} \label{figurateNumbers} The difference between the $n$-th $r$-generation hyperfibonacci number and the sum of its two consecutive predecessors is the $n$-th regular $(r-1)$-topic number,
\[
F_{n+2}^{(r)}=F_{n+1}^{(r)}+F_{n}^{(r)} + \binom{n+r}{r-1}, \enspace n \ge 0.
\]
\end{lemma}

\begin{proof} 
Again we use arguments on the last tile in board tilings. First we observe that tilings of an $(n+2r+1)$-board with at least $r$ dominoes, ending with a square are equinumerous to tilings of an $(n+2r)$-board with the same restriction. The latter has $F_{n+1}^{(r)}$ elements. 

We separate the tilings ending with a domino into two disjoint sets $A$ and $B$. The set $A$ consists of  tilings that have exactly $r$ dominoes and the set $B$ contains the rest of tilings, i.e., the tilings having at least $r+1$ dominoes. Having in mind that one domino is fixed, the tilings in the set $B$ are equinumerous to the tilings of an $(n+2r-1)$-board with the same restriction, i.e.,
\[
|B| = F_n^{(r)}.
\]
Now we use the fact that the number of tilings of an $m$-board with $M$ dominoes and $m-2M$ squares is equal to the number of $(m-M)$-combinations over a set of $M$ elements. 

A tiling in the set $A$ has $n+r$ tiles,
\begin{eqnarray*}
(n+2r+1)  - 2 - (r-1) = n+r,
\end{eqnarray*}
which means
\begin{eqnarray*}
|A|=\binom{n+r}{r-1}.
\end{eqnarray*}
The fact that 
\[
F_{n+2}^{(r)} = F_{n+1}^{(r)} + |A| + |B|
\]
completes the proof.
\end{proof}

Figure \ref{Fig2} illustrates this proof in the case $n=5$ and $r=2$.

Note that Lemma \ref{figurateNumbers} provides an equivalent definition of the hyperfibonacci sequences. For $n \ge 0$ and $r \ge 0$ we define the sequence $(F_n^{(r)})_{n \ge 0}$ by the recurrence relation
\[
F_{n+2}^{(r)} = F_{n+1}^{(r)}+F_{n}^{(r)} +P_{n+2}^{(r-1)}, 
\]
and the initial values $\enspace F_0^{(r)}=0, \enspace F_1^{(r)}=1$, where 
\[
P_{n+2}^{(r-1)}=\binom{n+r}{r-1}.
\]

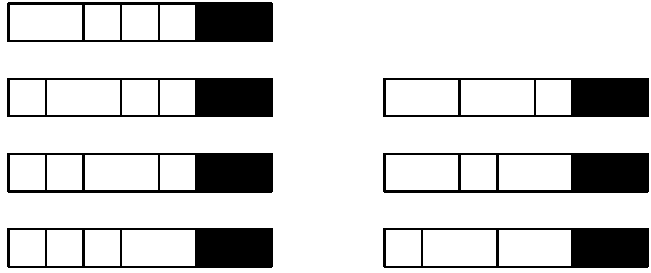
\begin{figure}[h]					
\setlength{\unitlength}{0.5cm}

\begin{picture}(0,10)(-7,0)
\put(0,0){\line(1,0){7}}
\put(0,1){\line(1,0){7}}
\put(0,2){\line(1,0){7}}
\put(0,3){\line(1,0){7}}
\put(0,4){\line(1,0){7}}
\put(0,5){\line(1,0){7}}
\put(0,6){\line(1,0){7}}
\put(0,7){\line(1,0){7}}

\put(0,0){\line(0,1){1}}
\put(7,0){\line(0,1){1}}
\put(0,2){\line(0,1){1}}
\put(7,2){\line(0,1){1}}
\put(0,4){\line(0,1){1}}
\put(7,4){\line(0,1){1}}
\put(0,6){\line(0,1){1}}
\put(7,6){\line(0,1){1}}

\put(1,0){\line(0,1){1}}
\put(2,0){\line(0,1){1}}
\put(3,0){\line(0,1){1}}
\put(5,0){\line(0,1){1}}

\put(1,2){\line(0,1){1}}
\put(2,2){\line(0,1){1}}
\put(4,2){\line(0,1){1}}
\put(5,2){\line(0,1){1}}

\put(1,4){\line(0,1){1}}
\put(3,4){\line(0,1){1}}
\put(4,4){\line(0,1){1}}
\put(5,4){\line(0,1){1}}

\put(2,6){\line(0,1){1}}
\put(3,6){\line(0,1){1}}
\put(4,6){\line(0,1){1}}
\put(5,6){\line(0,1){1}}

\put(11,0){\line(0,1){1}}
\put(13,0){\line(0,1){1}}
\put(15,0){\line(0,1){1}}

\put(12,2){\line(0,1){1}}
\put(13,2){\line(0,1){1}}
\put(15,2){\line(0,1){1}}

\put(12,4){\line(0,1){1}}
\put(14,4){\line(0,1){1}}
\put(15,4){\line(0,1){1}}

\put(10,0){\line(1,0){7}}
\put(10,1){\line(1,0){7}}
\put(10,2){\line(1,0){7}}
\put(10,3){\line(1,0){7}}
\put(10,4){\line(1,0){7}}
\put(10,5){\line(1,0){7}}

\put(10,0){\line(0,1){1}}
\put(17,0){\line(0,1){1}}
\put(10,2){\line(0,1){1}}
\put(17,2){\line(0,1){1}}
\put(10,4){\line(0,1){1}}
\put(17,4){\line(0,1){1}}

\multiput(5,0)(0.01,0){199} {\line(0,1){1}}
\multiput(5,2)(0.01,0){199} {\line(0,1){1}}
\multiput(5,4)(0.01,0){199} {\line(0,1){1}}
\multiput(5,6)(0.01,0){199} {\line(0,1){1}}

\multiput(15,0)(0.01,0){199} {\line(0,1){1}}
\multiput(15,2)(0.01,0){199} {\line(0,1){1}}
\multiput(15,4)(0.01,0){199} {\line(0,1){1}}

\end{picture}
\caption{Tilings in the sets $A$ (left) and $B$ (right) among the $F_4^{(2)}$ 7-board tilings. }\label{Fig2}
\end{figure}

\section{Identities for the $r$-th generation of hyperfibonacci numbers}

In this section we present some identities that hold for every generation of hyperfibonacci numbers. 

\begin{theorem} \label{HyperFiboBinom}
 For the $r$-th generation hyperfibonacci numbers 
\[
\sum_{k=r}^{\lfloor (n + 2r)/2 \rfloor} \binom{n+2r-k}{k} = F_{n+1}^{(r)}.
\]
\end{theorem}

\begin{proof}
According to \eqref{hyperFiboBoard}, the r.h.s. of the relation above is equal to $f_{n+2r}^{(r)}$. We use the fact that the number of $(n+2r)$-tilings with $k$ dominoes is equal to 
$\binom{n+2r-k}{k}$.
Having in mind that $f_{n+2r}^{(r)}$ represents the number of $(n+2r)$-tilings with at least $r$ dominoes we sum up the numbers of $(n+2r)$-tilings with $r, r+1, \ldots, \lfloor (n+2r)/2 \rfloor$ dominoes. 
\end{proof}

One of the basic properties of the hyperfibonacci sequence in the case $r=1$ is expressed by the relation
\begin{eqnarray} \label{FiboHyperFibo}
F_n^{(1)} = F_{n+2} - 1,
\end{eqnarray}
that is an immediate consequence of the elementary Fibonacci relation
\begin{eqnarray} \label{FiboSumFirstN}
1 + \sum_{k=0}^n F_k= F_{n+2}.
\end{eqnarray}
Consequently, an extension of the Binet formula to the first generation of  hyperfibonacci numbers is 
\begin{eqnarray} \label{BinetR1}
F_n^{(1)} = \frac{\phi^{n+2}  - \bar{\phi}^{n+2} }{\sqrt{5}} -1
\end{eqnarray}
where $ \phi$ and $\bar{ \phi}$ are the golden ratio and its conjugate, respectively, 
\begin{eqnarray*} \label{GoldenRatio}
 \phi = \frac{1+ \sqrt{5}}{2}, \enspace  \bar{ \phi} = \frac{1- \sqrt{5}}{2}.
\end{eqnarray*} 
Lemma \ref{HypefFiboF1SumFirstN} is an extension of the relation  \eqref{FiboSumFirstN} to the hyperfibonacci sequence for $r=1$.

\begin{lemma}\label{HypefFiboF1SumFirstN}
The sum of the first $n+1$ hyperfibonacci numbers of the first generation is equal to the difference between the $(n+4)$-th Fibonacci number and $(n+3)$,
\[
3 + n + \sum_{k=0}^n F_n ^{(1)} = F_{n+4}, \enspace n \ge 0.
\]
\end{lemma}

\begin{proof}
Using \eqref{FiboHyperFibo} we obtain
\begin{eqnarray*}
F_0^{(1)} + F_1^{(1)} + \cdots + F_n^{(1)} 
           &=& F_2 - 1 +F_3 -1 + \cdots + F_{n+2} - 1\\
           &=& \sum_{k=2}^{n+2} F_k - n -1\\
           &=& \sum_{k=0}^{n+2} F_k - n -2\\
           &=& F_{n+4} - n -3. 
\end{eqnarray*}
\end{proof}

Taking into account the definition of the hyperfibonacci numbers \eqref{hyperFibo}, Lemma \ref{HypefFiboF1SumFirstN} gives
\[
    F_{n}^{(2)} =   F_{n+4} - n -3,
\]
which can also be written as
\[
F_n^{(2)} = \frac{ \phi^{n+4}  - \bar{ \phi}^{n+4} }{\sqrt{5}} -n-3.
\]

Furthermore, we are going to show that the $n$-th hyperfibonacci number of the $r$-th generation $F_n^{(r)}$ is equal to the $(n+2r)$-th Fibonacci number diminished by the sum of $r$ binomial coefficients. This is expressed in Theorem \ref{BinetHyperfibo} and provides an extension of the Binet formula to hyperfibonacci sequences. The proof uses the known Fibonacci relation 
\begin{eqnarray} \label{FiboBinomElementary}
F_n= \sum_{k=0}^{\lfloor \frac{n-1}{2} \rfloor} \binom{n- k -1}{k},
\end{eqnarray}
that also follows from Theorem \ref{HyperFiboBinom} for $r=0$.

\begin{theorem} \label{BinetHyperfibo} The $n$-th hyperfibonacci number of $r$-th generation is equal to the difference between the $(n+2r)$-th Fibonacci number and the sum of $r$ binomial coefficients,
\[
F_n^{(r)}=F_{n+2r}-\sum_{k=0}^{r-1} \binom{n+r+k}{r-1-k}.
\]
\end{theorem}

\begin{proof}  We give a proof by induction on $m=n+r$. The induction base is as follows:

\begin{eqnarray*}
F_{n}^{(1)} &=& \displaystyle F_{n+2} -1 = F_{n+2\cdot 1}-\sum_{k=0}^{0} \binom{n+1-k}{k}\ \  (\mbox{for all $n$})\\
F_{n}^{(2)} &=& \displaystyle F_{n+4} -(n+3) = F_{n+2\cdot 2}-\sum_{k=0}^{1} \binom{n+3-k}{k}\ \  (\mbox{for all $n$}) \\
\end{eqnarray*}
Since by \eqref{FiboBinomElementary} we have
\begin{eqnarray*}
F_{2r} & =& \sum_{k=0}^{r-1} \binom{r+k}{r-1-k}\\
F_{2r+1} &=& 1+\sum_{k=0}^{r-1} \binom{r+1+k}{r-1-k},
\end{eqnarray*}
we also have (for all $r$)
\begin{eqnarray*}
F_{0}^{(r)} &= 0 = &F_{2r}-\sum_{k=0}^{r-1} \binom{r+k}{ r-1-k}\\
F_{1}^{(r)} &=  1 =& F_{2r+1}-\sum_{k=0}^{r-1} \binom{r+1+k}{r-1-k}.
\end{eqnarray*}

The induction step is as follows:
\begin{eqnarray*}
F_{n+1}^{(r)} &=&  F_{n}^{(r)} + F_{n+1}^{(r-1)}\\
 &=& F_{n+2r}-\sum_{k=0}^{r-1} \binom{n+r+k}{r-1-k} + F_{n+2r-1}-\sum_{k=0}^{r-2} \binom{n+r+k}{r-2-k}\\
 &=& F_{n+2r+1} - \binom{n+2r-1}{0}- \sum_{k=0}^{r-2}\left[ \binom{n+r+k}{r-1-k}+ 
 \binom{n+r+k}{r-2-k}\right]\\
 &=& F_{n+1+2r} - \binom{n+2r}{0}-\sum_{k=0}^{r-2} \binom{n+1+r+k}{r-1-k}\\
 & = &F_{n+1+2r} - \sum_{k=0}^{r-1} \binom{n+1+r+k}{r-1-k}.
\end{eqnarray*}
\end{proof}
The following result immediately follows from the relation \eqref{FiboBinomElementary}.
\begin{cor}
	%Using \eqref{FiboBinomElementary} immediately follows:
	\begin{eqnarray} \label{FiboBinomHyper}
	F_{n+1}^{(r)} =
	\sum_{k=0}^{\lfloor \frac{n}{2}\rfloor} \binom{n+r-k}{r+k}.
	\end{eqnarray}
\end{cor}
\begin{proof}
\begin{eqnarray*}
	F_{n+1}^{(r)} 
	&=& F_{n+1+2r} - \sum_{k=0}^{r-1} \binom{n+1+r+k}{r-1-k}=\\
	&=& \sum_{k=0}^{\left\lfloor \frac{n+1+2r-1}{2}\right\rfloor} \binom{n+1+2r-k-1}{k}
	- \sum_{k=0}^{r-1} \binom{n+2r-k}{k}=\\
	&=& \sum_{k=0}^{\lfloor \frac{n}{2}\rfloor+r} \binom{n+2r-k}{k}
	- \sum_{k=0}^{r-1} \binom{n+2r-k}{k}=\\
	&=& \sum_{k=r}^{\lfloor \frac{n}{2}\rfloor+r} \binom{n+2r-k}{k}=
	\sum_{k=0}^{\lfloor \frac{n}{2}\rfloor} \binom{n+r-k}{r+k}.
\end{eqnarray*}
\end{proof}

\begin{proposition}
\label{prop:LimitHyperPhi}
For any $r\ge 0$ the hyperfibonacci numbers of $r$-th generation  satisfy
\begin{equation*} 
\lim_{n\to \infty} \frac{F_{n+1}^{(r)}}{F_n^{(r)}}=\phi.
\end{equation*}
\end{proposition}

\begin{proof}
For the case $r=0$ we have the known property of Fibonacci numbers 
\begin{equation}
\label{LimitFiboPhi}
\lim_{n\to \infty}\frac{F_{n+1}}{F_n}=\phi.  
\end{equation}
\noindent
In the case $r=1$, one can prove the above relation as follows. By the formula \eqref{FiboHyperFibo}
 we have 
\begin{equation*}
\frac{F_{n+1}^{(1)}}{F_n^{(1)}}=\frac{F_{n+3}-1}{F_{n+2}-1}=\frac{\frac{F_{n+3}}{F_{n+2}}-\frac{1}{F_{n+2}}}{1-\frac{1}{F_{n+2}}}, 
\end{equation*}
and thus, since $\lim_{n \to \infty} \frac{1}{F_n}=0$, by using  \eqref{LimitFiboPhi}  we obtain $\lim_{n \to \infty}\frac{F_{n+1}^{(1)}}{F_{n}^{(1)}}=\phi$. For $r\ge 2$ it is convenient to apply Theorem \ref{BinetHyperfibo}. One can immediately see that the sum $\sum_{k=0}^{r-1} \binom{n+r+k}{r-1-k}$ occurring there is in fact the sum of $r$ polynomials in the variable $n$, of which $\binom{n+r}{r-1}=\frac{(n+r)(n+r+1)\dots(n+2)}{(r-1)!}$ has the highest degree, namely $r-1$. Analogously, by writing the formula in Theorem \ref{BinetHyperfibo} for $F_{n+1}^{(r)}$ we obtain
\begin{equation*}
\frac{F_{n+1}^{(r)}}{F_n^{(r)}}= \frac{\frac{F_{n+1+2r}}{F_{n+2r}}-\frac{Q_{r-1}(n)}{F_{n+2r}}}{1-\frac{R_{r-1}(n)}{F_{n+2r}}},
\end{equation*}
where $Q_{r-1}(n)$ and $R_{r-1}(n)$ are polynomials (with rational coefficients)  in the variable $n$ of degree $r-1$. By using the Binet formula for Fibonacci numbers, we have, for any polynomial $Q_{r-1}(n)$  in $n$ of degree $r-1$ (with $r\ge 1$ arbitrarily fixed),
\begin{equation*}
\displaystyle
\frac{Q_{r-1}(n)}{F_{n+2r}}
=\frac{Q_{r-1}(n)}{\frac{1}{\sqrt{5}}(\phi ^{n+2r}- \bar{\phi} ^{n+2r})} =\frac{\sqrt{5}}{\frac{\phi ^{n+2r}}{Q_{r-1}(n)}-\frac{\bar{\phi} ^{n+2r}}{Q_{r-1}(n)}}.
\end{equation*}
Since $|\phi|>1$ and $|\bar{\phi}|<1$, by basic mathematical analysis facts $\lim_{n\to \infty}\frac{\phi ^{n+2r}}{Q_{r-1}(n)}=\infty,$ $\lim_{n\to \infty}\frac{\bar{\phi} ^{n+2r}}{Q_{r-1}(n)}=0$, and we finally get  
$\lim_{n \to \infty}\frac{Q_{r-1}(n)}{F_{n+2r}}=0$.
This holds for both polynomials of degree $r-1$ that occur in the formula above and thus $\lim_{n \to \infty}\frac{F_{n+1}^{(r)}}{F_n^{(r)}}= \phi$.

\end{proof}

Let $l_m^{(r)}$ denote the number of $m$-{\it bracelet} tilings with squares and with at least $r$ dominoes. An $m$-bracelet tiling is formed from an $m$-board tiling by gluing together the cells $1$ and $n$. A bracelet is said to be {\it in phase} if it ends with a square or a domino. Otherwise, if a domino covers the cells $n$ and $1$, the bracelet is {\it out of phase}. In Lemma \ref{braceletHyperLucas}  we prove a relationship between bracelet tilings with constraint on the minimal number of dominoes and hyperlucas numbers. 

\begin{lemma} \label{braceletHyperLucas} Let $n,r \in \mathbb{N}_0$. Then the number of $(n+2r)$-bracelet tilings with squares and with at least $r$ dominoes is equal to the $n$-th hyperlucas number of the $r$-th generation,
\[
l_{n+2r}^{(r)} = L_{n} ^{(r)}.
\]
\end{lemma}
\begin{proof}
In the first step of the proof we show that the sequence of numbers $l_n^{(r)}$ obeys the same recurrence relation  \eqref{eqBasicRechyperFibo} as the hyperfibonacci numbers,
\begin{eqnarray*}
l_n^{(r)} = l_{n-1} ^{(r)} + l_n^{(r-1)}.
\end{eqnarray*}

We consider the last tile in an $(n+2r)$-bracelet tiling. If the $(n+2r)$-bracelet ends with a square, then the remaining $(n+2r-1)$-bracelet can be tiled in $l_{n+2r-1}^{(r)}$ ways. Otherwise, if it ends with a domino there are $l_{n+2r-2}^{(r-1)}$  possible tilings of the remaining $(n+2r-2)$-bracelet. Thus, bracelet tilings satisfy the same recurrence relation \eqref{eqBasicRechyperFibo} as hyperfibonacci numbers.

Now we are testing the initial condition. For $n=0$, there are two $2r$-bracelet tilings, one in phase and another one out of phase, with at least $r$ dominoes, thus $l_0^{(r)}=2$ and consequently $l_0^{(r)}=L_0^{(r)}$. For  $r=0$ there is no constraint on the number of dominoes and clearly we have $l_n^{(0)}= l_n=L_n$. This reasoning completes the proof.
\end{proof}

Theorem \ref{hyperfibohyperlucas} gives an analogue of the most elementary relation encountering both Fibonacci and Lucas numbers, that is 
\begin{eqnarray}\label{theMostElemFiboLucas}
L_n= F_{n-1} + F_{n+1}. 
\end{eqnarray}
When proving it, we use the combinatorial interpretation of hyperfibonacci numbers as the number of board tilings and the interpretation of hyperlucas numbers as the number of bracelet tilings.

\begin{theorem} \label{hyperfibohyperlucas}
The elements of the hyperfibonacci and the hyperlucas sequences satisfy the following formula:
\[
L_n^{(r)} =  F_{n-1}^{(r)} + F_{n+1} ^{(r)} +  \binom{n+r-1}{r-1}.
\]
\end{theorem}

\begin{proof}
Obviously, there is a one to one correspondence between the number of $(n+2r)$-board tilings and in phase bracelets of the same length. Let us separate the out of phase $(n+2r)$-bracelet tilings into two disjoint sets, such that the set $A$ contains tilings with exactly $r$ dominoes and the set $B$ contains tilings with at least $r-1$ dominoes.
Having in mind that one domino is fixed, the tilings in $B$ correspond to $(n+2r-2)$-board tilings with at least $r$ dominoes, thus 
\[
|B| = f_{n+2r-2}^{(r)} =  F_{n-1}^{(r)}.
\]

Since the tilings in $A$ consist of a number of $(n+2r-2)-(r-1)=n+r-1$  tiles, $r$ of them being dominoes (one of them is a fixed domino), we have
\[
|A| = \binom{n+r-1}{r-1}.
\]
The fact that 
\begin{eqnarray*}
L_{n}^{(r)} = F_{n+1}^{(r)}+ |A| + |B|
\end{eqnarray*}
completes the proof.
\end{proof}

Possibly, Theorem \ref{hyperfibohyperlucas} can be used to prove further hyperfibonacci - hyperlucas identities.

\section{Some identities for the first generation}

In this section we demonstrate that there are Cassini-like formulas for the first generation of hyperfibonacci sequences. We were able to prove both Theorem \ref{CassiniHyperfibo} and Theorem \ref{CatalanHyperfibo} using the extension \eqref{BinetR1} of the Binet formula to the hyperfibonacci sequence in the case $r=1$. We also provide elegant proofs employing the Cassini identity \cite{Voll, WeZe} for the Fibonacci sequence,
\begin{eqnarray} \label{CassiniId}
F_{n-1}F_{n+1} - F_n^2 = (-1)^n, \enspace n \ge 0
\end{eqnarray}
and its generalization 
\begin{eqnarray} \label{CatalanId}
 F_n^2 - F_{n-r}F_{n+r}= (-1)^{n-r} F_r^2, \enspace n \ge 0, \enspace r \in \mathbb{N}_0
\end{eqnarray}
known as the Catalan identity.

In particular, we are going to show that for even $n$ the difference of the square of the $n$-th hyperfibonacci number $F_n^{(1)}$ and the product of its two neighbors is equal to $F_{n-3}^{(1)}$,  
\begin{eqnarray*}
{F_n^{(1)}}^2 - F_{n-1}^{(1)} F_{n+1}^{(1)} &=& F_{n-3}^{(1)}. 
\end{eqnarray*}
Otherwise, when $n$ is odd the relation 
\begin{eqnarray*}
{F_n^{(1)}}^2 - F_{n-1}^{(1)} F_{n+1}^{(1)}&=& F_{n-3}^{(1)} +2
\end{eqnarray*}
holds. We prove these facts in Theorem \ref{CassiniHyperfibo} by means of the Cassini identity. 

\begin{theorem} \label{CassiniHyperfibo} For the first generation of hyperfibonacci numbers and $n \ge 3$ the following identity holds:
\[
{F_n^{(1)}}^2 - F_{n-1}^{(1)} F_{n+1}^{(1)} = F_{n-3}^{(1)} + 1 + (-1)^{n+1}.
\]
\end{theorem}

\begin{proof}
Using \eqref{FiboHyperFibo} and \eqref{CassiniId}
we have
\begin{eqnarray*}
{F_{n}^{(1)}}^2-F_{n-1}^{(1)}F_{n+1}^{(1)} &=& (F_{n+2}-1)^2-(F_{n+1}-1)(F_{n+3}-1)\\
& = & F_{n+2}^2-2F_{n+2}+1 - (F_{n+1}F_{n+3}-F_{n+1}-F_{n+3}+1)\\
&= &F_{n+2}^2-F_{n+1}F_{n+3}+(F_{n+3}-F_{n+2})-(F_{n+2}-F_{n+1})\\
&= &-(-1)^{n+2}+F_{n+1}-F_n \\
&=&(-1)^{n+1}+F_{n-1}\\
&= &1+(-1)^{n+1}+F_{n-1}-1\\
&=&1+(-1)^{n+1}+F_{n-3}^{(1)}.
\end{eqnarray*}
\end{proof}
 
There is also regularity in the difference between the square of a hyperfibonacci number of the first generation and the product of its two second neighbors. Namely, according to the closed formula for the first generation hyperfibonacci numbers \eqref{BinetR1} we get
\begin{eqnarray*}
& &{F_n^{(1)}}^2 - F_{n-2}^{(1)} F_{n+2}^{(1)}  = \\
&& \Bigg (\frac{   \phi^{n+2} - \bar{ \phi}^{n+2}  }{\sqrt{5}} - 1 \Bigg )^2 - \Bigg (\frac{   \phi^{n} - \bar{ \phi}^{n}  }{\sqrt{5}} - 1 \Bigg ) \Bigg (\frac{   \phi^{n+4} - \bar{ \phi}^{n+4}  }{\sqrt{5}} - 1 \Bigg )\\
&=& \frac{  ( \phi^{n+2} - \bar{ \phi}^{n+2} )^2 }{{5}}  - \frac{ 2(  \phi^{n+2} - \bar{ \phi}^{n+2} ) }{\sqrt{5}} + 1 -\frac{1}{5} \Big [   \phi^{2n+4} + \bar{ \phi}^{2n+4} -(-1)^{n}( \phi^4 + \bar{ \phi}^4)   \\
& &- \sqrt{5}(  \phi^{n+4} - \bar{ \phi}^{n+4})  - \sqrt{5}( \phi^{n} - \bar{ \phi}^{n})         +5    \Big ]\\
&=& (-1)^{n} + \frac{ \phi^{n+4} - \bar{ \phi}^{n+4}  }{\sqrt{5}}   + \frac{ \phi^{n} - \bar{ \phi}^{n}  }{\sqrt{5}} - 2 \frac{ \phi^{n+2} - \bar{ \phi}^{n+2}  }{\sqrt{5}}
\end{eqnarray*}
which is equal to 
\begin{eqnarray*}
(-1)^{n} + F_{n+3} + F_{n+2} + F_n - 2(F_{n+1} + F_n)\\
\end{eqnarray*}
and finally to
\begin{eqnarray*}
(-1)^{n}  + F_{n+2} &=& (-1)^{n} +1 + F_n^{(1)}.
\end{eqnarray*}
This result is expressed in Theorem \ref{CatalanHyperfibo}. We also provide a proof by means of the Catalan identity.

\begin{theorem} \label{CatalanHyperfibo} For the first generation of  hyperfibonacci numbers and $n \ge 2$
\[
{F_n^{(1)}}^2 - F_{n-2}^{(1)} F_{n+2}^{(1)} = F_{n}^{(1)} + 1 + (-1)^n.
\]
\end{theorem}
 
\begin{proof} Using \eqref{FiboHyperFibo} and \eqref{CatalanId} we have
\begin{eqnarray*}
{F_{n}^{(1)}}^2 - F_{n-2}^{(1)}F_{n+2}^{(1)} &=& (F_{n+2}-1)^2-(F_{n}-1)(F_{n+4}-1)\\
&=&F_{n+2}^2-2F_{n+2}+1 - (F_{n}F_{n+4}-F_{n}-F_{n+4}+1)\\
&=&F_{n+2}^2-F_{n}F_{n+4}+(F_{n+4}-F_{n+2})-(F_{n+2}-F_{n})\\
&=&(-1)^{n}F_2^2+F_{n+3}-F_{n+1}\\
&=&(-1)^{n}+F_{n+2}\\
&=&1+(-1)^{n}+F_{n}^{(1)}. 
\end{eqnarray*}
\end{proof}

%\noindent
\section{Acknowledgments} 
We thank the referees for valuable comments that helped us improve the quality of the paper. 
\\
Ligia L. Cristea is founded by the Austrian Science and Research Fund FWF, Project P27050-N26.

\noindent {2000 \it AMS Mathematical Subject Classifications:} Primary 05A17; Secondary 11P84.\\
{\it Keywords:} Fibonacci sequence, hyperfibonacci sequence, hyperlucas sequence, Binet formula, polytopic number.\\
(Concerned with sequences A000045, A000071, A001924, A014162, A014166.)

\end{document}